\documentclass[12pt]{amsart}
\usepackage{amsmath,amsthm,amsfonts,amssymb,mathrsfs}
\date{\today}

\usepackage{color}
\input xymatrix
\xyoption{all}

\usepackage{hyperref}
\usepackage{todonotes}

\newtheorem{theorem}{Theorem}[section]

\newtheorem{proposition}[theorem]{Proposition}

\theoremstyle{definition}
\newtheorem{example}[theorem]{Example}
\newtheorem{remark}[theorem]{Remark}

\begin{document}

\title[On old and new classes of feebly compact spaces]{On old and new classes of feebly compact spaces}

\author[O.~Gutik]{Oleg~Gutik}
\address{Faculty of Mathematics, National University of Lviv,
Universytetska 1, Lviv, 79000, Ukraine}
\email{oleg.gutik@lnu.edu.ua,
ovgutik@yahoo.com}

\author[A.~Ravsky]{Alex~Ravsky}
\address{Pidstryhach Institute for Applied Problems of Mechanics and Mathematics of NASU, Naukova 3b, Lviv, 79060, Ukraine}
\email{alexander.ravsky@uni-wuerzburg.de}



\begin{abstract}
We introduce three new classes of countably pracompact spaces, consider their basic properties
and relations with other compact-like spaces.
\end{abstract}

\maketitle

\bigskip

\section{Definitions and relations}

In general topology we often investigate different classes of compact-like spaces
and relations between them, see, for instance, basic \cite[Chap. 3]{Engelking1989} and
general works \cite{vanDouwenReedRoscoeTree1991}, \cite{Matveev1998},
\cite{VaughanHSTT}, \cite{StephensonJr1984}, \cite{Lipparini2016}.
We consider the present paper as a next small step in this quest.

We shall follow the terminology of \cite{Engelking1989}. By $\mathbb{N}$ we shall denote the set of all positive integers.

A subset of a topological space $X$ is called \emph{regular open} if it equals the interior of its closure.
A space $X$ is \emph{quasiregular} if each nonempty open subset of $X$ contains closure of some nonempty open subset of $X$.

\subsection{Old classes} We recall that a topological space $X$ is said to be

\begin{itemize}
  \item \emph{semiregular} if $X$ has a base consisting of regular open subsets;
  \item \emph{compact} if each open cover of $X$ has a finite subcover;
  \item \emph{sequentially compact} if each sequence $\{x_n\}_{n\in\mathbb{N}}$ of $X$ has a convergent subsequence in $X$;
  \item \emph{$\omega$-bounded} if each countable subset of $X$ has compact closure;
  \item \emph{totally countably compact} if each sequence of $X$ contains a subsequence with
  compact closure;
  \item \emph{countably compact} if each open countable cover of $X$ has a finite subcover;
  \item \emph{countably compact at a subset} $A\subseteq X$ if every infinite subset $B\subseteq A$  has  an  accumulation  point $x$ in $X$;
  \item \emph{countably pracompact} if there exists a dense subset $D$ in $X$  such that $X$ is countably compact at $D$;
  \item \emph{feebly $\omega$-bounded} if for each sequence $\{U_n\}_{n\in\mathbb{N}}$  of non-empty  open  subsets of $X$ there is a compact subset $K$ of
  $X$ such that $K\cap U_n\ne\varnothing$ for each $n$;
  \item \emph{selectively sequentially feebly compact} if for each sequence $\{U_n\}_{n\in\mathbb{N}}$ of non-empty  open  subsets of $X$ we  can  choose  a  point $x_n\in U_n$ for each $n\in\mathbb N$ such that the sequence $\{x_n\}$ has a convergent subsequence;
  \item \emph{selectively feebly compact}\footnote{Selectively sequentially feebly compact Tychonoff spaces were recently introduced and studied by Dorantes-Aldama and Shakhmatov in \cite{Dorantes-AldamaShakhmatov}. Also they considered selectively feebly compact Tychonoff spaces under the name \emph{selectively pseudocompact} spaces.
An equivalent property appeared a few years earlier in papers by Garc\i a-Ferreira with Ortiz-Castillo~\cite{Garcia-FerreiraOrtiz-Castillo2014} and with Tomita~\cite{Garcia-FerreiraTomita2015} under the title ``\emph{strong pseudocompactness}'', but since the term ``strongly pseudocompact'' is used  in~\cite{ArkhangelskiiGenedi1993, Dikranjan1994} to denote two different properties, we stick to a name for this property which reflects its ``selective'' nature and also matches the name of the previous ``selective'' property.}, if for each sequence $\{U_n\}_{n\in\mathbb{N}}$ of non-empty  open  subsets of $X$ we  can  choose  a  point  $x\in X$ and  a  point $x_n\in U_n$ for  each $n\in\mathbb N$ such that the set $\{n\in\mathbb N\colon x_n\in W\}$ is  infinite  for  every  open neighborhood $W$ of $x$.
  \item \emph{sequentially feebly compact}~\cite[Def. 1.4]{DowPorterStephensonWoods2004} if for each sequence $\{U_n\colon n\in\mathbb{N}\}$ of non-empty open subsets of the
space $X$ there exist a point $x\in X$ and an infinite set $I\subset\mathbb{N}$ such that for each neighborhood $U$ of the point $x$ the set $\{n\in I\colon U_n\cap U=\varnothing\}$ is finite;
\footnote{One of the authors introduced this notion a few years ago as a natural property intermediate between feeble and sequential compactness, which may be useful in some applications in topological algebra. Indeed, for instance, Proposition 1.10. by Artico et al.~\cite{ArticoMarconiPelantRotterTkachenko} combined with Theorem 1.1 by Lipparini~\cite{Lipparini2016} states that that each $T_0$ feebly compact topological group is sequentially feebly compact. But later we found that it is a known property, even with the same name. The oldest reference which we know (see  \cite[p. 15]{Matveev1998}) is Reznichenko's paper~\cite{Reznichenko}. A similar notion had been given by Artico et al. in~\cite[Def. 1.8]{ArticoMarconiPelantRotterTkachenko}, where are used pairwise disjoint open sets instead. Lipparini proved in~\cite{Lipparini2016} that these notions are equivalent.}
  \item \emph{feebly compact} if each locally finite family of nonempty open subsets of the space $X$
  is finite.
  \item $k$-\emph{space} if $X$ is Hausdorff and a subset $F\subset X$ is closed in $X$
  if and only if $F\cap K$ is closed in $K$ for every compact subspace $K\subset X$.
\end{itemize}
According to Theorem~3.10.22 of \cite{Engelking1989}, a Tychonoff topological space $X$ is
feebly compact if and only if it is pseudocompact, that is,
each continuous real-valued function on $X$ is bounded.
Also, a Hausdorff topological space $X$ is feebly compact if and only if every locally
finite family of non-empty open subsets of $X$ is finite.

Relations between different classes of compact-like spaces are well-studied. Some of them
are presented at Diagram 3 in \cite[p.17]{Matveev1998}, at Diagram 1 in~\cite[p.~58]{Dorantes-AldamaShakhmatov} (for Tychonoff spaces),
and at Diagram~3.6 in~\cite[p. 611]{StephensonJr1984}.

\subsection{New classes} The notion of countable pracompactness has
been studied by several authors under several names. According to Matveev~\cite{Matveev1998}
it ``appeared in the literature under many different names''. Matveev
mentions that Baboolal, Backhouse and Ori \cite{Baboolal-Backhouse-Ori-1990} introduced an
equivalent notion under the name $e$-countable compactness.
In the recent paper \cite{Martinez-Cadena--Wilsoni-2017}  the authors study the notion using the expression ``densely countably compact''. A few references and a further name are recalled
there \cite{Arkhangelskii1996} 
According to Arkhangel'skii~\cite{Arkhangelskii1989}
countable compactenss at some subset and countable pracompactness
``find important applications in $C_p$-theory''.

In order to refine the stratification of countable pracompact spaces even more,
we introduce the following definitions. In each of them we require that
a space $X$ contains a dense subset $D$ with a special property. Namely,

\begin{itemize}
  \item if each sequence of points of the set $D$ has a convergent subsequence (in $X$)
then $X$ is \emph{sequentially pracompact};
  \item if each sequence of points of the set $D$ has a subsequence with compact closure (in
  $X$) then $X$ is \emph{totally countably pracompact};
  \item if each countable subset of the set $D$ has compact closure (in $X$)
  then $X$ is \emph{$\omega$-bounded-pracompact}.
\end{itemize}

Our main motivation to introduce the above spaces is possible applications
in topological algebra. In particular, we are going to use them in the paper \cite{Gutik-Sobol-2018}.

The following diagram shows relations between different classes of compact-like spaces.
All implications on the diagram are true and we \emph{suggest} that
they are well-known or easy to prove and
all non-marked arrows and not reversible without imposing additional
conditions on spaces. In particular,
in Section \ref{Sec:Examples} of the present paper we construct a
sequentially feebly compact space which is
not selectively feebly compact (Example~\ref{Hedgehog}),
a sequentially pracompact space which is not countably compact (Example~\ref{Example2}),
and a totally countably pracompact space which is nether $\omega$-bounded-pracompact
nor not totally countably compact (Example~\ref{Example3}).

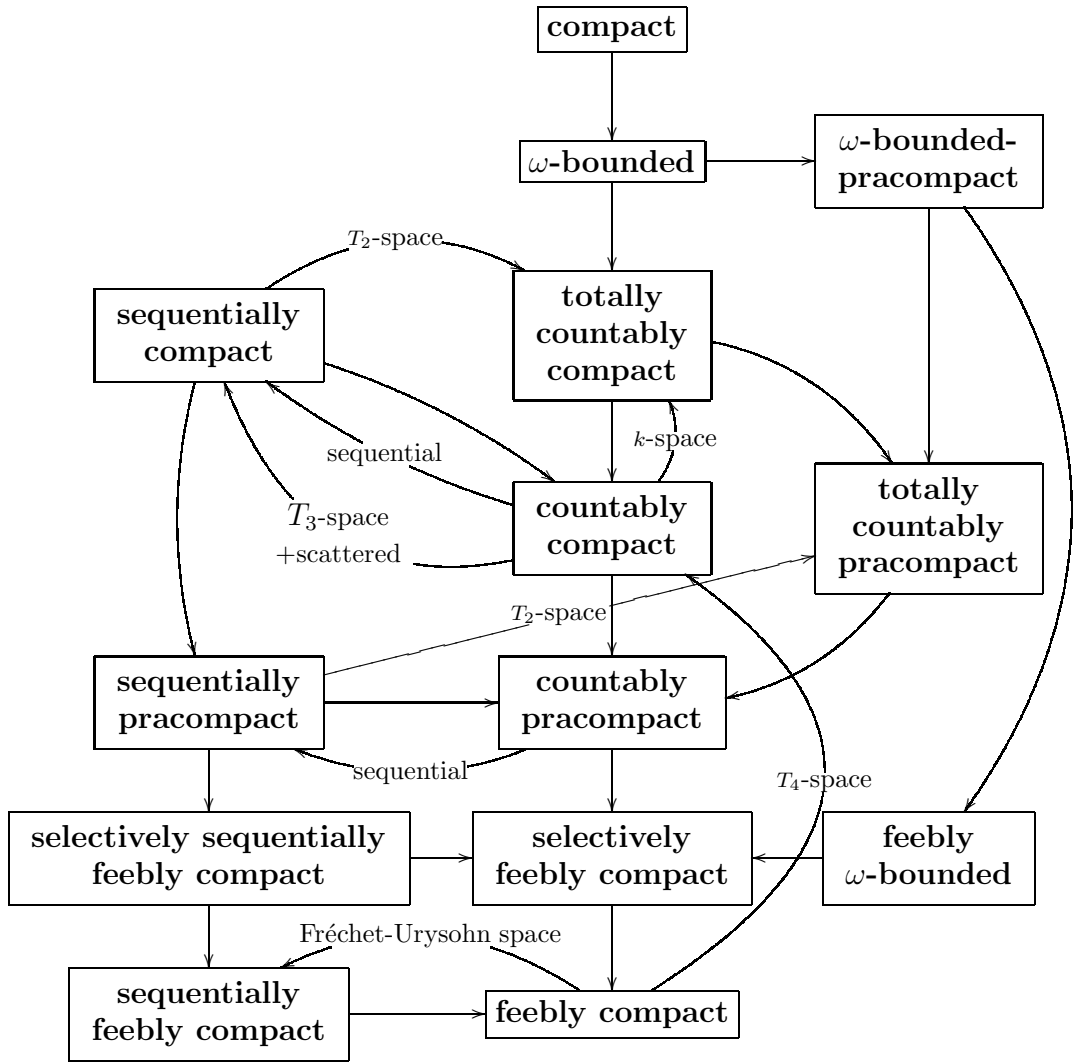
\begin{equation*}
\xymatrix{
&*+[F]{\hbox{\textbf{compact}}}\ar[d] &\\
 &*+[F]{\hbox{\textbf{$\omega$-bounded}}}\ar[d]\ar[r] &
 *+[F]{\begin{array}{c}\hbox{\textbf{$\omega$-bounded-}}\\ \hbox{\textbf{pracompact}}\end{array}}\ar[dd]\ar@/^4.5pc/[dddd]\\
*+[F]{\begin{array}{c}
                                                          \hbox{\textbf{sequentially}}\\
                                                          \hbox{\textbf{compact}}
                                                        \end{array}}\ar@/^3pc/[r]|-{\footnotesize{T_2\hbox{-space}}}\ar@/_1pc/[dd]\ar@/^1pc/[rd] &*+[F]{\begin{array}{c}
                                                          \hbox{\textbf{totally}} \\
                                                          \hbox{\textbf{countably}}\\
                                                          \hbox{\textbf{compact}}
                                                        \end{array}
}\ar[d]\ar@/^2pc/[dr] &\\
 &*+[F]{\begin{array}{c}
                                                          \hbox{\textbf{countably}}\\
                                                          \hbox{\textbf{compact}}
                                                        \end{array}}\ar[d]\ar@/_2pc/[u]|-{\footnotesize{k\hbox{-space}}}
                                                        \ar@/^4pc/[lu]|-{\begin{array}{c}
                                                          \footnotesize{T_3\hbox{-space}}\\
                                                          \footnotesize{\hbox{+scattered}}
                                                        \end{array}} \ar@/^1pc/[lu]|-{\hbox{\footnotesize{sequential}}} &
 *+[F]{\begin{array}{c}
                                                          \hbox{\textbf{totally}}\\
                                                          \hbox{\textbf{countably}}\\
                                                          \hbox{\textbf{pracompact}}
                                                        \end{array}}\ar@/^2pc/[dl]\\
*+[F]{\begin{array}{c}
                                                          \hbox{\textbf{sequentially}}\\
                                                          \hbox{\textbf{pracompact}}
                                                        \end{array}}\ar[d]\ar[r]\ar[rru]|-{\footnotesize{\!T_2\hbox{-space}~}~}
 &*+[F]{\begin{array}{c}
                                                          \hbox{\textbf{countably}}\\
                                                          \hbox{\textbf{pracompact}}
                                                        \end{array}}\ar[d]\ar@/^2.2pc/[l]|-{\hbox{\footnotesize{sequential}}} &  \\
*+[F]{\begin{array}{c}
                                                          \hbox{\textbf{selectively sequentially}}\\
                                                          \hbox{\textbf{feebly compact}}
                                                        \end{array}}\ar[r]\ar[d]&
*+[F]{\begin{array}{c}                                   \hbox{\textbf{selectively}}\\
                                                          \hbox{\textbf{feebly compact}}\end{array}}\ar[d]  &
*+[F]{\begin{array}{c}                                   \hbox{\textbf{feebly}}\\
                                                          \hbox{\textbf{$\omega$-bounded}}\end{array}}\ar[l]  &\\
*+[F]{\begin{array}{c}
                                                          \hbox{\textbf{sequentially}}\\
                                                          \hbox{\textbf{feebly compact}}
                                                        \end{array}}\ar[r]& *+[F]{\hbox{\textbf{feebly compact}}}\ar@/_2.5pc/[l]|-{\hbox{\footnotesize{Fr\'{e}chet-Urysohn space}}} \ar@/_6.7pc/[uuu]|-{\footnotesize{T_4\hbox{-space}}} &   &\\
&&
}
\end{equation*}

\section{Basic properties}

\subsection{Extensions} We recall that an \emph{extension} of a space $X$ is a space
$Y$ containing $X$ as a dense subspace.
It is easy to check that countable pracompactness,
sequential pracompactness, feeble compactness, sequential feeble compactness,
selective feeble compactness, selective sequential feeble compactness, and
feeble $\omega$-boundedness is preserved by extensions.

\subsection{Continuous images}
 
It is easy to check that sequential compactness, feeble compactness,
sequential feeble compactness, countably pracompactness, and sequential pracompactness is preserved by continuous
images and total countable compactness, total countable pracompactness,
$\omega$-boundedness, and $\omega$-bounded-pracompactness is preserved by continuous Hausdorff images.

\subsection{Products}

The investigation of productivity of compact-like spaces
is motivated by the fundamental Tychonoff theorem,
stating that a product of a family of compact spaces is compact,
On the other hand, there are two countably compact
spaces whose product is not feebly compact (see~\cite{Engelking1989},
the paragraph before Theorem 3.10.16). The product of a countable family of sequentially
compact spaces is sequentially compact~\cite[Theorem~3.10.35]{Engelking1989}.
But already the Cantor cube $D^\frak c$ is not sequentially compact
(see~\cite{Engelking1989}, the paragraph after Example 3.10.38).
On the other hand some compact-like spaces are also preserved by products, see
~\cite[$\S$ 3-4]{VaughanHSTT} (especially Theorem 3.3, Proposition 3,4, Example 3.15, Theorem 4.7, and Example 4.15)
and $\S$7 for the history, and~\cite[$\S$ 5]{StephensonJr1984}.
Among more recent results we note that
Dow et al. in Theorem 4.1 of \cite{DowPorterStephensonWoods2004} proved that
a product of a family of sequentially feebly compact spaces is
again sequentially feebly compact, and in Theorem 4.3 that every product of feebly compact
spaces, all but one of which are sequentially feebly compact, is feebly compact.

In the next propositions we show that sequentially pracompact, $T_1$ totally countably pracompact, and $\omega$-bounded-pracompact spaces
are preserved by products.
The proofs are easy and straightforward but we provide them because a theorem should have a proof.

Let $X$ be a product of a family $\{X_\alpha\colon \alpha\in A\}$ of spaces.
For each
subset $B$ of the set $A$ as $\pi_B$ we denote the projection from
$X=\prod\{X_\alpha\colon \alpha\in A\}$ to $\prod\{X_\alpha\colon \alpha\in B\}$.
If $B=\{\alpha\}$ then $\pi_B$ we shall denote also as $\pi_\alpha$.
A space $Y\subset X$ is called a $\Sigma$-product of the family $\{X_\alpha\}$
provided there exists a point $y\in X$ such that
$Y=\{x\in X:x_\alpha=y_\alpha$ for all but countably many $\alpha\in A\}$.
In this case $Y$ is also called the \emph{Corson $\Sigma$-subspace of $X$ based at $y$}.

\begin{proposition}
The ($\Sigma$-) product of a family of sequentially pracompact spaces is
sequentially pracompact.
\end{proposition}
\begin{proof}
Let $X$ be the non-empty product of a family
$\{X_\alpha\colon \alpha\in A\}$ of sequentially pracompact spaces
and $Y\subset X$ be the Corson $\Sigma$-subspace of $X$ based at a point $y=(y_\alpha)\in X$.
For each index $\alpha\in A$ fix a dense subset
$D_\alpha\ni y_\alpha$ of the space $X_\alpha$ such that
each sequence of points of the set $D_\alpha$ has a convergent subsequence
and fix a point $a_\alpha\in D_\alpha$.
Put $D=Y\cap \prod_{\alpha\in A} D_\alpha $.
Then the set $D$ is a dense subset of the space $X$.
Let $C=\{x_n\colon n\in\mathbb{N}\}$ be a sequence of points of the
set $D$ and
$B=\{\alpha_m\colon m\in\mathbb{N}\}$ be an enumeration of the countable set
$\{\alpha\in A:\exists x\in C(x_\alpha\ne y_\alpha)\}$.
By induction we can build a sequence $\{x_{\alpha_m}\in X_{\alpha_m}\}$ of points and
a sequence $\{S_m\}$ of {\it infinite} subsets of $\mathbb{N}$ such that
$S_m\supset S_{m'}$ for each $m\le m'$ and
for each neighborhood $U_{\alpha_m}\subset X_{\alpha_m}$ of the point $x_{\alpha_m}$ the set
$\{n\in S_m\colon x_{n\alpha_m}\not\in U_{\alpha_m}\}$ is finite.
We can easily construct an infinite set $S\subset\mathbb{N}$ such that
the set $S\setminus S_m$ is finite for each $m\in\mathbb{N}$.
Choose a point $x=(x_\alpha)\in Y$ such that $x_\alpha$ is already
defined for $\alpha\in B$ and $x_\alpha=y_\alpha$ for $\alpha\in A\setminus B$.
Let $U$ be an arbitrary neighborhood of the point $x$.
There exist a finite subset $F$ of the set $A$
and a family $\{U_\alpha\colon\alpha\in F, U_\alpha\subset X_\alpha$ is an open neighborhood
of $x_\alpha\}$ such that
$x\in U'=\pi_F^{-1}\left(\prod\{U_\alpha\colon \alpha\in F\}\right)\subset U$. The
inductive construction implies that the set
$T_\alpha=\{n\in S\colon x_{n\alpha}\not\in U_{\alpha}\}$ is finite for each
$\alpha\in F$. Then $x_n\in U'\subset U$ for each
$n\in S\setminus\bigcup\{T_\alpha\colon \alpha\in F\}$.
\end{proof}

\begin{proposition}\label{Prod3}
The ($\Sigma$-) product of a family of totally countably pracompact $T_1$ spaces is
totally countably pracompact.
\end{proposition}
\begin{proof}
Let $X$ be the non-empty product of a family
$\{X_\alpha\colon \alpha\in A\}$ of totally countably pracompact spaces
and $Y\subset X$ be the Corson $\Sigma$-subspace of $X$ based at a point $y=(y_\alpha)\in X$.
For each index $\alpha\in A$ fix a dense subset
$D_\alpha\ni y_\alpha$ of the space $X_\alpha$ such that
each sequence of points of the set $D_\alpha$ has a
subsequence with compact closure in $X_\alpha$.
Put $D=Y\cap \prod_{\alpha\in A} D_\alpha $.
Then the set $D$ is a dense subset of the space $X$.
Let $C=\{x_n\colon n\in\mathbb{N}\}$ be a sequence of points of the
set $D$ and
$\{\alpha_m\colon m\in\mathbb{N}\}$ be an enumeration of the countable set
$\{\alpha\in A:\exists x\in C(x_\alpha\ne y_\alpha)\}$.
By induction we can build
a sequence $\{S_m\}$ of {\it infinite} subsets of $\mathbb{N}$ such that
$S_m\supset S_{m'}$ for each $m\le m'$ and the set $\{x_{n\alpha_m}:n\in S_m\}$
has compact closure in $X_{\alpha_m}$.
We can easily construct an infinite set $S\subset\mathbb{N}$ such that
the set $S\setminus S_m$ is finite for each $m\in\mathbb{N}$.
Then the set $\{x_{n}:n\in S\}$ has compact closure in $X$,
which is contained in $Y$.
\end{proof}

\begin{remark} The referee remarked that in Cartesian product case of
Proposition~\ref{Prod3} $T_1$ condition can be a weakened to
that for each $\alpha\in A$ a set $\{y_\alpha\}$ has compact closure in $X_\alpha$.
The proof remain almost the same, only the final words "which is contained in $Y$"
have to be dropped.

It motivates to define a class of spaces in
which every singleton (that is, one-point set) has compact closure.
The referee suggested to investigate which classes of compact-like space
belong to the class.
By definition, each $T_1$ space belong to the class.
Each totally countably compact space $X$ also belongs to the class because
for any point $x\in X$ a set $\overline{\{x\}}$ is closure of any subsequence of a
constant sequence $\{x_n\}$, where $x_n=x$ for each $n$.

On the other hand, the referee proposed to endow
$\omega$ with the topology of left intervals,
whose open sets are the intervals $[0, n)$, plus
the whole of $\omega$. Here closure of $0$ is the noncompact $\omega$.
We extend this construction as follows. Let $X=\omega_1+\omega$ endowed with
a topology with a subbase consisting of halfintervals $[0,\alpha)$, where
$\alpha<\omega_1+\omega$ and $(\alpha,\omega_1+\omega)$,  where $\alpha<\omega_1$.
Then closure of $\omega_1$ is a noncompact set $[\omega_1,\omega_1+\omega)$.
Now put $D=\omega_1$. Then $D$ is dense in $X$ and
each countable subset $C$ of $D$ is contained in a closed compact set $[0,\sup C]$ of
$D$. Thus $X$ is both sequentially and $\omega$-bounded-pracompact.

A sequentially compact example of a space not belonging to the class is more complicated,
but, luckily, already known. Namely, in~\cite[Example 5]{RavskyPPG} the second author built
a group $G=\bigoplus\limits_{\alpha\in\omega_1} \mathbb Z$, which is the direct sum of
the groups $\mathbb Z$
and its subgroup $$S=\{0\}\cup\{(x_\alpha)\in G:(\exists \beta\in\omega_1)( (\forall \alpha>\beta)
(x_\alpha=0)\&(x_\beta>0))\}.$$
Let $G_S$ be the group $G$ endowed with a topology with a base $\{x+S:x\in G\}$. Then $G_S$
is a paratopological group, that is the group operation $+:G\times G\to G$ is
continuous. In~\cite[Example 5]{RavskyPPG} it is shown that the group $G_S$ is
sequentially compact. On the other hand, by~\cite[Lemma 17]{RavskyPPG} the set
$S\subset G_S$ is compact. Since $\overline{\{0\}}=\{x\in G:x+S\ni 0\}=-S$,
if the set $-S$ is compact then $G=S\cup (-S)$ is compact too, which contradicts
~\cite[Proposition 12]{RavskyPPG}.
\end{remark}

\begin{proposition}\label{Prod4}
The product of a family of $\omega$-bounded-pracompact spaces is $\omega$-bounded-pracompact.
Moreover, if all spaces of the family are $T_1$ then a $\Sigma$-product of the family is
$\omega$-bounded-pracompact too.
\end{proposition}
\begin{proof}
Let $X$ be the non-empty product of a family
$\{X_\alpha\colon \alpha\in A\}$ of $\omega$-bounded-pracompact spaces
and $Y\subset X$ be the Corson $\Sigma$-subspace of $X$ based at a point $y=(y_\alpha)\in X$.
For each index $\alpha\in A$ fix a dense subset
$D_\alpha\ni y_\alpha$ of the space $X_\alpha$ such that
each countable subset of the set $D_\alpha$ has compact closure in $X_\alpha$.
Put $D=Y\cap \prod_{\alpha\in A} D_\alpha $.
Then the set $D$ is a dense subset of the space $X$.
Let $C$ be a countable subset of the set $D$.
Then $C$ is a subset of a closed compact subset
$C'=\prod_{\alpha\in A}\overline{\pi_\alpha(C)}$ of the space $X$.
Now assume that all spaces $X_\alpha$ are $T_1$.
Put $B=\{\alpha\in A:\exists x\in C(x_\alpha\ne y_\alpha)\}$. The set $B$ is
countable and so $C'=\prod_{\alpha\in B} \overline{\pi_\alpha(C)}
\times \prod_{\alpha\in A\setminus B} \{y_\alpha\}\subset Y$.
\end{proof}

\begin{example} This example shows that $T_1$ condition is essential in $\Sigma$-product case of
Propositions~\ref{Prod3} and ~\ref{Prod4}.
Let $X'$ be a space consisting of two distinct points $a$ and $b$
endowed with a topology consisting of sets $\varnothing$, $\{a\}$, and $X'$.
Let $A$ be an uncountable subset, $X$ be the
product of a family $\{X_\alpha\colon \alpha\in A\}$,
$Y\subset X$ be the Corson $\Sigma$-subspace of $X$ based at a point $y=(a_\alpha)\in X$,
where $X_\alpha=X'$ and $a_\alpha=a$ for each $\alpha\in A$.
Since the space $X'$ is compact, it is easy to check that the space $Y$ is countably compact.
On the other hand, the space $Y$ is not totally countably pracompact.
For this purpose it suffices to show that for
any point $x=(x_\alpha)\in Y$ a set $\overline{\{x\}}$ (everywhere in this
example we by $\overline{S}$ we mean closure in $Y$ of its subset $S$)
is not compact , because $\overline{\{x\}}$
is closure (in $Y$) of any subsequence of a constant sequence $\{x_n\}$,
where $x_n=x$ for each $n$. By~\cite[Proposition 2.3.3]{Engelking1989},
$\overline{\{x\}}=\overline{\{(x_\alpha)\}}=Y\cap \prod_{\alpha\in A}\overline{\{x_\alpha\}}$.
Remark that $b\in \overline{\{x_\alpha\}}$ for each $\alpha\in A$.
Now for each $\alpha\in A$ put $Y_\alpha=\{y=(y_\beta)\in Y:y_\alpha=a\}$. Since
for each point $z=(z_\alpha)\in Y$ there exists an index $\alpha$ such that $z_\alpha=a$,
the family $\{Y_\alpha:\alpha\in A\}$ is an open cover of the set $Y$, and hence of $\overline{\{x\}}$.
Let $C$ be any finite subset of $A$. Let $t=(t_\alpha)\in Y$ be such that
$t_\alpha=b$ if $x_\alpha=b$ or $\alpha\in C$ and $t_\alpha=a$, otherwise.
Then $t\in \overline{\{x\}}\setminus\bigcup \{Y_\alpha:\alpha\in C\}$.
Thus the set $\overline{\{x\}}$ is not compact.
\end{example}

Since sequential feebly compactness is preserved by extensions, the next proposition
strengthen a bit Theorem 4.1 of \cite{DowPorterStephensonWoods2004}.

\begin{proposition} The $\Sigma$-product of a family of sequentially feebly compact spaces is
sequentially feebly compact.
\end{proposition}
\begin{proof}
Let $X$ be a non-empty product of a family $\{X_\alpha\colon \alpha\in A\}$ of sequentially
feebly compact spaces,
$Y\subset X$ be the Corson $\Sigma$-subspace of $X$ based at a point $y=(y_\alpha)\in X$,
and $\{V_n\colon n\in\mathbb{N}\}$ be a sequence of non-empty open subsets of the space $Y$.
For each index $n$ choose a finite subset $B_n$ of the set $A$ and a family
$\{U_{n\alpha}\colon \alpha\in B_n, U_{n\alpha}$ is a non-empty open subset of $X_\alpha\}$
such that $U_n\cap Y \subset V_n$, where $U_n=\pi_{B_n}^{-1}\left(\prod\{U_{n\alpha}\colon \alpha\in
B_n\}\right)$. Put $B=\bigcup B_n$. By Theorem 4.1 of \cite{DowPorterStephensonWoods2004},
the space $X'=\{X_\alpha\colon \alpha\in B\}$ is sequentially feebly compact. Since
$\{\pi_B(U_n)\}$ is a sequence of its non-empty open subsets,
there exist a point $x'\in X'$ and an infinite set $I\subset\mathbb{N}$
such that for each neighborhood $U'$ of the point $x'=(x'_\alpha)_{\alpha\in B}$ the set
$\{n\in I\colon \pi_B(U_n)\cap U'=\varnothing\}$ is finite.
Define a point $x=(x_\alpha)_{\alpha\in A}\in Y$ by putting $x_\alpha=x'_\alpha$ for each
$\alpha\in B$ and $x_\alpha=y_\alpha$ for each $\alpha\in A\setminus B$.
Let $V$ be an arbitrary neighborhood of the point $x$ in the space $Y$.
Pick a canonical neighborhood $U$ of the point $x$ in the space $X$ such that $U\cap Y\subset V$.
Then there exists a subset $I'$ of the set $I$ such that a set $I\setminus I'$ is finite
and $\pi_B(U_n)\cap \pi_B(U)\ne\varnothing$ for each $n\in I'$. Fix any such $n$ and
pick a point $z'=(z'_\alpha)_{\alpha\in B}\in \pi_B(U_n)\cap \pi_B(U)$.
Define a point $z=(z_\alpha)_{\alpha\in A}\in Y$ by putting $z_\alpha=z'_\alpha$ for each
$\alpha\in B$ and $z_\alpha=y_\alpha$ for each $\alpha\in A\setminus B$.
It is easy to check that $z\in U_n\cap U\cap Y\subset V_n\cap V$.
\end{proof}

\section{Backward implications}

Banakh and Zdomskyy in~\cite{BanakhZdomskyy2004} defined a topological space $X$ to be an
\emph{$\alpha_7$-space}
if for any family $\{S_n:n\in\mathbb{N}\}$ of countable infinite subsets of the space $X$
such that a set $S_n\setminus U$ is finite for any $n$ and any neighborhood
$U$ of $x$ there exist a countable infinite subset $S$ of the space $X$ and
a point $y\in X$ such that a set $S\setminus V$ is finite for any neighborhood
$V$ of $y$ and $S_n\cap S\ne\varnothing$ for infinitely many $n$.

\begin{proposition}
Let $X$ be a Fr\'echet-Urysohn feebly compact space. Then $X$ is sequentially feebly compact.
Moreover, if $X$ is quasiregular or $\alpha_7$ then $X$ is
selectively sequentially feebly compact.
\end{proposition}
\begin{proof}
Let $X$ be a Fr\'echet-Urysohn feebly compact space and
$\{V_n\colon n\in\mathbb{N}\}$ be a sequence of non-empty open subsets of the space $X$.
For each $n$ choose a non-empty open set $U_n\subset V_n$ such that $\overline{U_n}\subset V_n$
provided the space $X$ is quasiregular.
Since the space $X$ is feebly compact, there exists a point $x\in X$ such that each neighborhood
of the point $x$ intersects infinitely many sets of the sequence $\{U_n\}$.
Put $I_0=\left\{n\in\mathbb{N}\colon x\in\overline{U_n}\right\}$.

Suppose the set $I_0$ is infinite. Then $U\cap U_n\ne\varnothing$ for each $n\in I_0$
and each neighborhood $U$ of the point
$x$. If the space $X$ is quasiregular then $x\in V_n$ for each $n\in I_0$, so a
constant sequence $\{x_n=x:n\in I_0\}$ converges to $x$. Assume that $X$ is an $\alpha_7$-space.
Since the space $X$ is Fr\'echet-Urysohn, for each $n\in I_0$ there exists
a sequence $S'_n=\{x^n_k:k\in\mathbb{N}\}$ of points of $U_n$ convergent to a point $x$. Considering its subsequence,
if necessarily, we can assume that the sequence $S'_n$ consists of distinct points
or it is constant. In the latter case we have $x^n_k=x^n\in U_n$ for each $k$
for some point $x^n\in U_n$ such that $x\in \overline{\{x^n\}}$. Put
$I'_0=\{n\in I_0:S'_n$ is constant$\}$.
If the set  $I'_0$ is infinite then a sequence $\{x^n:n\in I'_0\}$ converges to
the point $x$. So we suppose the set $I'_0$ is finite.
Since $X$ is an $\alpha_7$-space,
there exist a countable infinite subset $S$ of the space $X$ and
a point $y\in X$ such that a set $S\setminus V$ is finite for any neighborhood
$V$ of $y$ and a set
$I''_0=\{n\in I_0\setminus I'_0:$ there exists a natural $k(n)$ such that $x^n_{k(n)}\in S\}$
is infinite. For each $n\in I''_0$ put $x_n=x^n_{k(n)}\in U_n$. If
there exists a point $z\in X$ such that a set $I_1=\{n\in I''_0:x_n=z\}$ is infinite then
a sequence $\{x_n:n\in I_1\}$ converges to the point $z$. Otherwise
a sequence $\{x_n:n\in I''_0\}$ converges to
the point $y$. Indeed, let $V$ be an arbitrary
neighborhood of the point $y$. Then a set $S\setminus V$ is finite
and $x_n\in V$ for each $n\in I''_0\setminus\{n: x_n\in S\setminus V\}$.

Suppose the set $I_0$ is finite. Since
$x\in\overline{\bigcup\left\{U_n\colon n\in\mathbb{N}\setminus I_0\right\}}$ and $X$ is a
Fr\'echet-Urysohn space, there exists a sequence $\{x'_m\colon m\in\mathbb{N}\}$ of points of
the set $\bigcup\{U_n\colon n\in\mathbb{N}\setminus I_0\}$ converging to the point $x$.
For each index $m\in\mathbb{N}$ choose an index $n(m)\in\mathbb{N}\setminus I_0$ such that
$x'_m\in U_{n(m)}$. Put $I_1=\{n(m)\colon m\in\mathbb{N}\}$. Since $x\not\in\overline{U_n}$
for each $n\in\mathbb{N}\setminus I_0$, the set $I_1$ is infinite.For each $r\in I_1$ pick a point
$x_r=x'_{m(r)}$, where $n(m(r))=r$. Then $x_r\in U_r$ and a sequence $\{x_r:r\in I_1\}$ converges to
the point $x$. Indeed, let $U$ be an arbitrary
neighborhood of the point $x$. Since the sequence $\{x'_m\}$ converges to the point $x$,
there exists $N\in\mathbb{N}$ such that $x'_m\in U$ for each $m>N$. Then
$x_r\in U$ for each $r\in I_1\setminus\left\{n(m)\colon 0\le m\le N\right\}$.
\end{proof}

\begin{proposition}
Each sequential countably pracompact space is sequentially pracompact.
\end{proposition}

\begin{proof} Let $X$ be a sequential countably pracompact space. There
exists a dense subset $D$ of the space $X$ such that
each infinite subset of the set $D$ has an accumulation point in $X$.
Let $\{x_n\colon n\in\mathbb{N}\}$ be a sequence of points of the
set $D$. If there exists a point $x\in X$ such that $x\in\overline{\{x_n\}}$ for infinitely many
indices $n\in\mathbb{N}$ then the $\{x_n:x_n=x\}$ is a
convergent subsequence of the sequence $\{x_n\colon n\in\mathbb{N}\}$.
So we suppose that there is no such point $x$. Then the set
$B=\{x_n\colon n\in\mathbb{N}\}$ is infinite.
The set $B$ has an accumulation point $y$ in $X$. Then $y\in\overline{B\setminus\{y\}}$.
Therefore the set $B\setminus\{y\}$ is not sequentially closed and there exists a sequence
$\{z_m\colon m\in\mathbb{N}\}$ of points of the set $B\setminus\{y\}$ converging to a
point $z\not\in B\setminus\{y\}$. Then the sequence
$\{z_m\colon m\in\mathbb{N}\}$ contains infinitely many different points of the set
$B\setminus\{y\}$.
\end{proof}

\begin{proposition} Each countably pracompact $k$-space $X$ is totally countably pracompact.
\end{proposition}
\begin{proof} There exists a dense subset $D$ of the space $X$ such that
each infinite subset of the set $D$ has an accumulation point in $X$.
Let $\{x_n\colon n\in\mathbb{N}\}$ be a sequence of points of the
set $D$. Put $B=\{x_n\colon n\in\mathbb{N}\}$. If the set $B$ is finite then
there exists a point $x\in X$ such that $x_n=x$ for infinitely many
indices $n\in\mathbb{N}$. Then a subsequence $\{x_n:x_n=x\}$
of the sequence $\{x_n\colon n\in\mathbb{N}\}$ has compact closure $\{x\}$ in
$X$. So we suppose that the set $B$ is infinite.
The set $B$ has an accumulation point $y$ in $X$.
Then $y\in\overline{B\setminus\{y\}}$.
Therefore the set $B\setminus\{y\}$ is not closed
and there exists a compact subset $K$ of the space $X$ such that a set $B\cap K$ is
not closed in $K$. Then the set $B\cap K$ is infinite, the sequence
$\{x_n:x_n\in B\cap K\}$ is infinite too and $\overline{\{x_n:x_n\in B\cap K\}}\subset K$.
\end{proof}

\begin{proposition} Each sequentially feebly compact space containing a dense set
$D$ of isolated points is sequentially pracompact.
\end{proposition}
\begin{proof} It is easy to check that each sequence of points of the set $D$
has a convergent subsequence.
\end{proof}

\section{Examples}\label{Sec:Examples}

\begin{example}\label{Hedgehog}
Let $X_0$ be a
non-empty $T_1$ space. Determine a topology on the set
$X=(X_0\times\omega)\cup\{y_0\}$, where $y_0\not\in X_0\times\omega$ by the following base

$$\mathscr{B}= \left\{U\times \{n\}\colon U\hbox{~is an open subset of the space~}X_0, n\in\omega
\right\}\cup$$
$$    \bigcup\left\{\{y_0\}\cup\bigcup_{m\ge n}X_0\times \{m\}\setminus F_m\colon n\in\omega, F_m \hbox{~is a finite subset of~} X_0 \hbox{~for each~}m\in\omega \hbox{~such that~} m\ge
n\right\}.$$

It is easy to check the following:
\begin{itemize}
  \item the space $X$ is Hausdorff provided the space $X_0$ is Hausdorff;
  \item the space $X$ is feebly compact provided the space $X_0$ is a feebly compact space without isolated points;
  \item the space $X$ is sequentially feebly compact provided the space $X_0$ is a sequentially feebly compact space without isolated points.
\end{itemize}

Now we take the standard unit segment $[0,1]$ as $X_0$. Then $X$ is a sequentially feebly compact
space, containing a closed discrete infinite subspace $\{1\}\times\omega$.
Now for each $n\in\omega$ put $U_n=X_0\times\{n\}$. Let $\{x_n\}$ be a sequence of
points of the space $X$ such that $x_n\in U_n$. Then the set $\{x_n\}$ has no accumulation points,
so the space $X$ is not selectively feebly compact.
\end{example}

We recall that the Stone-\v{C}ech compactification of a Tychonoff space $X$ is a
compact Hausdorff space $\beta X$ containing $X$ as a dense subspace so that each continuous map
$f\colon X\rightarrow Y$ to a compact Hausdorff space $Y$
extends to a continuous map $\overline{f}\colon \beta X\rightarrow Y$ \cite{Engelking1989}.

\begin{example}\label{Example2} (see ~\cite[Exer. 3.6.I]{Engelking1989}, cf.~\cite[Ex. 2.6]{Dorantes-AldamaShakhmatov}.) Let $\{N_\alpha\}_{\alpha\in A}$,
where $A\cap\mathbb N=\varnothing$, be an infinite family of infinite subsets of $\mathbb N$
such that the intersection $N_\alpha\cap N_\beta$ is finite for every pair $\alpha,\beta$ of
distinct elements of $A$ and that $\{N_\alpha\}_{\alpha\in A}$ is maximal with respect to
the last property. Generate a topology on the set $X=\mathbb N\cup S$ by the neighborhood
system $\{\mathcal B(x)\}_{x\in X}$, where $\mathcal B(x)=\{\{n\}\}$, if $x=n\in\mathbb N$
and $\mathcal B(x)=\{\{\alpha\}\cup (N_\alpha\setminus\{1,2,\dots,n\})\}_{n=1}^\infty$ if
$x=\alpha\in A$.

Since $A$ is a closed discrete infinite subset of $X$, $X$ is not countably compact.
On the other hand, the set $D=\Bbb N$ is dense in $X$. Let $\{x_n\colon n\in\mathbb{N}\}$
be an arbitrary sequence of points of the set $D$. If the \emph{set} $S=\{x_n\colon n\in\mathbb{N}\}$
is finite then the sequence $\{x_n\colon n\in\mathbb{N}\}$ has a constant subsequence.
If the set $S$ is infinite then by maximality of $A$ there exists $\alpha\in A$ such that
$N_\alpha\cap S$ is infinite. Note that the enumeration $\{x_{n_k}: k\in\mathbb N\}$ of $N_\alpha\cap S$
in the increasing order is a subsequence of the sequence
$\{x_n\colon n\in\mathbb{N}\}$ converging to the point $\alpha$.
Thus the space $X$ is sequentially pracompact.
\end{example}

\begin{example}\label{Example3}
Endow the set $\mathbb{N}$ with the discrete topology.
Let $\mathscr{A}(\mathbb{N})=\mathbb{N}\cup\{\infty\}$ be a one-point Alexandroff compactification
of $\mathbb{N}$ with the remainder $\infty$. We define on $\mathscr{A}(\mathbb{N})\times\mathbb{N}$
the product topology $\tau_p$ and extend the topology $\tau_p$ onto
$X=\mathscr{A}(\mathbb{N})\times\mathbb{N}\cup\{a\}$,
where $a\notin\mathscr{A}(\mathbb{N})\times\mathbb{N}$, to a topology $\tau^*$ in the following way:
bases of the topologies $\tau_p$ and $\tau^*$ coincide at $x$ for any $x\in\mathscr{A}(\mathbb{N})\times\mathbb{N}$ and the family
\begin{equation*}
    \mathscr{B}^*(a)=\left\{U_a(i_1,\ldots,i_n)\colon i_1,\ldots,i_n\in\mathbb{N}\right\},
\end{equation*}
where
\begin{equation*}
    U_a(i_1,\ldots,i_n)=X\setminus\left(\left(\{\infty\}\times\mathbb{N}\right)\cup\left(\mathscr{A}(\mathbb{N})\times\left\{ i_1,\ldots,i_n\right\}\right)\right),
\end{equation*}
determines a  set of neighbourhood systems for $\tau^*$ at the point $a$.

The definition of the topology $\tau^*$ on $X$ implies that $\mathbb{N}\times\mathbb{N}$ is the maximum discrete subspace of $(X,\tau^*)$ and $\mathbb{N}\times\mathbb{N}$ is dense in $(X,\tau^*)$. Hence every dense subset $D$ of $(X,\tau^*)$ contains $\mathbb{N}\times\mathbb{N}$. But $\overline{\mathbb{N}\times\mathbb{N}}=X$ is not compact, and hence $(X,\tau^*)$ is not an $\omega$-bounded-pracompact space.

Now we shall show that $(X,\tau^*)$ is totally countably pracompact. Especially we shall prove that $\mathbb{N}\times\mathbb{N}$ is the requested dense subset of the space $(X,\tau^*)$. Fix an arbitrary sequence $\left\{x_n\right\}_{n\in\mathbb{N}}\subset \mathbb{N}\times\mathbb{N}$. If there exists a positive integer $i$ such that the set $\left\{x_n\right\}_{n\in\mathbb{N}}\cap\left(\mathscr{A}(\mathbb{N})\times\left\{ i\right\}\right)$ is infinite then the subsequence $\left\{x_{i_j}\right\}_{j\in\mathbb{N}}=\left\{x_n\right\}_{n\in\mathbb{N}}\cap\left(\mathscr{A}(\mathbb{N})\times\left\{ i\right\}\right)$ with the corresponding renumbering has compact closure in $(X,\tau^*)$. In the other case the set $\left\{x_n\right\}_{n\in\mathbb{N}}\cap\left(\mathscr{A}(\mathbb{N})\times\left\{ i\right\}\right)$ is finite for any  positive integer $i$. Then the definition of $(X,\tau^*)$ implies that $\overline{\left\{x_n\right\}_{n\in\mathbb{N}}}=\{a\}\cup\left\{x_n\right\}_{n\in\mathbb{N}}$ is a compact subset of $(X,\tau^*)$.

We observe that by Proposition~19 of \cite{Gutik-Pavlyk=2013a}, $(X,\tau^*)$ is Hausdorff non-semiregular countably pracompact non-countably compact space, and hence $(X,\tau^*)$ is not totally countably compact.
\end{example}

\section*{Acknowledgements}

The authors thank to Paolo Lipparini for references and to the referee for more references and
valuable remarks and suggestions.

\end{document}